\documentclass[12pt,leqno]{article}
\usepackage{amsmath,amssymb,bbm}
\usepackage{amsthm,amscd}
\def\bm#1{\mathbbm{#1}}
\def\fn#1{\mathop{{\rm #1}\vphantom{\dim}}}

\makeatletter
\renewcommand{\section}{\@startsection{section}{1}{0mm}{12mm}{5mm}{\raggedright\bf\large}}

\def\@citex[#1]#2{\if@filesw\immediate\write\@auxout{\string\citation{#2}}\fi
  \def\@citea{}\@cite{\@for\@citeb:=#2\do
    {\@citea\def\@citea{\@citesep}\@ifundefined
       {b@\@citeb}{{\bf ?}\@warning
       {Citation `\@citeb' on page \thepage \space undefined}}%
{\csname b@\@citeb\endcsname}}}{#1}}
\def\@citesep{; }
\makeatother

\newtheoremstyle{Kang}{}{}{\itshape}{}{\bf}{}{.5em}{}
\theoremstyle{Kang}
\newtheorem{theorem}{Theorem}[section]

\newtheorem{lemma}[theorem]{Lemma}
\newtheorem{coro}[theorem]{Corollary}

\newtheoremstyle{Kremark}{}{}{}{}{\bf}{}{.5em}{}
\theoremstyle{Kremark}
\newtheorem{defn}[theorem]{Definition}
\newtheorem{example}[theorem]{Example}
\newtheorem*{remark}{Remark.}
\newtheorem{other}{}
\newenvironment{idef}[1]{\renewcommand{\theother}{#1}\begin{other}}{\end{other}}

\allowdisplaybreaks[1]  

\title{Retract Rational Fields}
\author{Ming-chang Kang\\[2mm]
Department of Mathematics and \\ Taida Institute of Mathematical Sciences\\
National Taiwan University\\
Taipei, Taiwan\\
E-mail: kang@math.ntu.edu.tw}
\date{}

\begin{document}

\maketitle

\footnote{\hspace*{-7.5mm}
Partially supported by National Center for Theoretic Science
(Taipei office).}

\begin{abstract}
{\noindent\bf Abstract} Let $k$ be an infinite field. The notion
of retract $k$-rationality was introduced by Saltman in the study
of Noether's problem and other rationality problems. We will
investigate the retract rationality of a field in this paper.
Theorem 1. Let $k\subset K\subset L$ be fields. If $K$ is retract
$k$-rational and $L$ is retract $K$-rational, then $L$ is retract
$k$-rational. Theorem 2. For any finite group $G$ containing an
abelian normal subgroup $H$ such that $G/H$ is a cyclic group, for
any complex representation $G \rightarrow GL(V)$, the fixed field
$\bm{C}(V)^G$ is retract $\bm{C}$-rational. Theorem 3. If $G$ is a
finite group, then all the Sylow subgroups of $G$ are cyclic if
and only if $\bm{C}_{\alpha}(M)^G$ is retract $\bm{C}$-rational
for all $G$-lattices $M$, for all short exact sequences $\alpha :
0 \rightarrow \bm{C}^{\times} \rightarrow M_{\alpha} \rightarrow M
\rightarrow 0$. Because the unramified Brauer group of
 a retract $\bm{C}$-rational field is trivial, Theorem 2 and
Theorem 3 generalize previous results of Bogomolov and Barge
respectively (see Theorem \ref{t5.9} and Theorem \ref{t6.1}).
\end{abstract}

\section{Introduction}

Let $k$ be a field, and $L$ be a finitely generated field
extension of $k$. $L$ is called $k$-rational (or rational over
$k$) if $L$ is purely transcendental over $k$, i.e.\ $L$ is
isomorphic to some rational function field over $k$. $L$ is called
stably $k$-rational if $L(y_1,\ldots,y_m)$ is $k$-rational for
some $y_1,\ldots,y_m$ which are algebraically independent over
$L$. $L$ is called $k$-unirational if $L$ is $k$-isomorphic to a
subfield of some $k$-rational field extension of $k$. It is easy
to see that ``$k$-rational" $\Rightarrow$ ``stably $k$-rational"
$\Rightarrow$ ``$k$-unirational".

Let $G$ be a finite group acting on the rational function field
$k(x_g:g\in G)$ by $k$-automorphisms defined by $h\cdot
x_g=x_{hg}$ for any $g,h\in G$. Denote by $k(G)$ the fixed
subfield, i.e. $k(G)=k(x_g:g\in G)^G$. Noether's problem asks,
under what situation, the field $k(G)$ is $k$-rational.

Note that, if $k$ is an infinite field and $k(G)$ is $k$-rational
(resp. stably $k$-rational), then there exists a generic
$G$-Galois extension over $k$ \cite[Theorem 5.1]{Sa2}. On the
other hand, when Hilbert's irreducibility theorem is valid for $k$
(e.g. if $k$ is any algebraic number field), it is not difficult
to see that the existence of a generic $G$-Galois extension over
$k$ implies that there is a Galois field extension $K$ over $k$
such that $Gal(K/k) \simeq G$, i.e. the inverse Galois problem for
the pair $(k,G)$ is solvable (see, for example, \cite[Theorem
3.3]{Sw1}). In the study of generic Galois extensions and generic
division algebras, Saltman was led to the notion of retract
$k$-rationality \cite{Sa1,Sa4}, which is the main subject of this
paper.

\begin{defn}[{\cite[p.130; Sa4, Definition 3.1]{Sa1}}] \label{d1.1}
Let $k$ be an infinite field and $L$ be a field containing $k$.
$L$ is called retract $k$-rational, if there are some affine
domain $A$ over $k$ and $k$-algebra morphisms $\varphi: A\to
k[X_1,\ldots,X_n][1/f]$, $\psi:k[X_1,\ldots,X_n][1/f]$ where
$k[X_1,\ldots,X_n]$ is a polynomial ring over $k$, $f\in
k[X_1,\ldots,X_n]\backslash \{0\}$, satisfying that

(i) $L$ is the quotient field of $A$, and

(ii) $\psi \circ \varphi = 1_A$, the identity map on $A$.
\end{defn}

In the above definition of retract $k$-rationality, it is required
that $k$ is an infinite field because this assumption guarantees
the existence of sufficiently many $k$-specializations when we
apply the notion of retract rationality to other concepts or
problems. Here is a geometric picture of retract rationality.
Suppose that $L$ is retract $k$-rational over $k$. Then there are
quasi-projective varieties $V$ and $W$ defined over $k$, a
dominating $k$-morphism $p: V \rightarrow W$ satisfying that
$k(W)=L$, $k(V)$ is $k$-rational and $p$ has a section, i.e. a
$k$-morphism $s: W \rightarrow V$ with $ps=1_W$.

Another related notion is discussed by Colliot-Th\'{e}l\`{e}ne and
Sansuc \cite{CTS3}. A field $L$ over $k$ is called a direct factor
of a $k$-rational field if there is a field $L'$ over $k$ such
that the quotient field of $L\otimes_{k}L'$ is $k$-rational (in
particular, the $k$-algebra $L\otimes_{k}L'$ is an integral
domain). It is known that, if $L$ is the function field of some
algebraic torus $T$ over $k$, then $L$ is retract $k$-rational if
and only if it is a direct factor of some $k$-rational field
\cite[Proposition 7.4]{CTS3}.

Return to Noether's problem.

\begin{theorem}[\cite{Sa2,Sa4,De}] \label{t1.2}
Let $k$ be an infinite field and $G$ be a finite group.
The following statements are equivalent,

{\rm (i)} $k(G)$ is retract $k$-rational;

{\rm (ii)} There is a generic $G$-Galois extension over $k$;

{\rm (iii)} There exists a generic $G$-polynomial over $k$.
\end{theorem}

\begin{proof}
${\rm (i)}\Leftrightarrow{\rm (ii)}$ by \cite[Theorem 5.3; Sa4,
Theorem 3.12]{Sa2}. The equivalence of (i), (ii), (iii) was proved
in \cite{De,DM}.
\end{proof}

It is not difficult to verify that, if $k$ is an infinite field,
then ``$k$-rational" $\Rightarrow$ ``stably $k$-rational"
$\Rightarrow$ ``retract $k$-rational" $\Rightarrow$
``$k$-unirational". Thus, if $k(G)$ is not retract $k$-rational,
then $k(G)$ is not stably $k$-rational (and is not $k$-rational,
in particular). This is the strategy for showing that $\bm{C}(G)$
is not $\bm{C}$-rational for some group $G$ of order $p^9$ by
Saltman in \cite{Sa3} (where $p$ is any prime number). On the
other hand, if $k(G)$ is $k$-rational, then $k(G)$ is retract
$k$-rational.

We remark that the direction of the implication ``rational"
$\Rightarrow$ ``stably rational" $\Rightarrow$ ``retract rational"
$\Rightarrow$ ``$k$-unirational" cannot be reversed. There is a
field extension $L$ of $\bm{C}$ such that $L$ is stably
$\bm{C}$-rational, but not $\bm{C}$-rational \cite{BCTSSD}. If
$C_p$ denotes the cyclic group of order $p$, then $\bm{Q}(C_p)$ is
retract $\bm{Q}$-rational, but not stably $\bm{Q}$-rational when
$p=47$, 113 or 233, etc. (see  Theorem \ref{t3.7} and the remark
after its proof). $\bm{Q}(C_8)$ is $\bm{Q}$-unirational, but not
retract $\bm{Q}$-rational (see Theorem \ref{t2.9}); for finitely
generated field extensions over $\bm{C}$ which are
$\bm{C}$-unirational, but not retract $\bm{C}$-rational, see
\cite{Sa3,Bo,CHKK}. On the other hand, we don't know whether there
is a field extension $L$ of $\bm{C}$ such that $L$ is retract
$\bm{C}$-rational, but is not stably $\bm{C}$-rational. The reader
is referred to the papers \cite{MT,CTS2} for surveys of the
rationality problems, and to Swan's paper \cite{Sw1} for Noether's
problem.

\bigskip
In this paper, we will prove a transitivity theorem for retract
rationality in Theorem \ref{t4.2}. Then we will show that
$\bm{C}(V)^G$ is retract $\bm{C}$-rational where $G \rightarrow
GL(V)$ is any complex representation and $G$ is a finite group
containing an abelian normal subgroup $H$ such that $G/H$ is a
cyclic group (see Theorem \ref{t5.10}). Because of Theorem
\ref{t3.2}, Theorem \ref{t5.10} may be regarded as a
generalization of a result of Bogomolov (see Theorem \ref{t5.9}).
Finally we will show that, if $G$ is a finite group, then all the
Sylow subgroups of $G$ are cyclic if and only if
$\bm{C}_{\alpha}(M)^G$ is retract $\bm{C}$-rational for all
$G$-lattices $M$, for all short exact sequences $\alpha : 0
\rightarrow \bm{C}^{\times} \rightarrow M_{\alpha} \rightarrow M
\rightarrow 0$. This result generalizes a theorem of Barge (see
Theorem \ref{t6.1}).

An application of the transitivity theorem is Theorem \ref{t5.4},
which asserts that $k(G)$ is retract $k$-rational is equivalent to
the retract $k$-rationality of $k(M)^G$ where $M$ is any faithful
$G$-lattice with $[M]^{fl}$ invertible. We remark that Theorem
\ref{t2.9} and Theorem \ref{t3.7}, due to Voskresenskii and
Saltman respectively, are of interest themselves. The proof of
these two theorems are included for the convenience of the reader.

\bigskip
We remark that there is a notion, called the property
$\fn{Rat}(G/k)$ by Serre \cite[p.86]{GMS}, which is slightly
stronger than the existence of a generic $G$-Galois extension over
$k$. We define it as follows.
\begin{defn}[{\cite[p.11,86]{GMS}}] \label{d1.3}
Let $k$ be an infinite field and $G$ be a finite group. We say
that the property $\fn{Rat}(G/k)$ holds for the pair ($G, k$), if
there exists a versal $G$-torsor over $L$ where $L$ is some
$k$-rational field extension.
\end{defn}

In order to explain this property, we define first the notion of a
$G$-Galois covering.

\begin{defn}[{\cite[p.43; Mi2, p.41; Sw1, Proposition 2.1]{Mi1}}] \label{d1.4}
Let $G$ be a finite group. Let $R \subset S$ be commutative rings
such that the group $G$ acts on $S$ by $R$-automorphisms of $S$
with $R=S^G$ where $S^G$ is the ring of invariants of $S$ under
the action of $G$. We say that $S$ is a Galois covering of $R$
with group $G$ (for short, $S$ is a $G$-Galois covering of $R$),
if the morphism $h: S \otimes_R S \rightarrow \prod_{\sigma \in
G}S$ defined below is an isomorphism where we define $h(s_1
\otimes s_2)=(\cdots, h_{\sigma}(s_1 \otimes s_2),
\cdots)_{\sigma}\in \prod_{\sigma \in G}S$ with $h_{\sigma}(s_1
\otimes s_2)=s_1\cdot\sigma (s_2)$ (i.e. the $\sigma$-th
coordinate of $h(s_1 \otimes s_2)$ is $s_1 \cdot \sigma (s_2)$).
We also say that $Spec (S) \rightarrow Spec (R)$ is a $G$-Galois
covering if $S$ is a $G$-Galois covering of $R$.

The above definition can be globalized. Namely, when $V, W$ are
schemes or algebraic varieties defined over a field $k$ and $V
\rightarrow W$ is a faithfully flat morphism, we can define by the
similar way the notion that $V \rightarrow W$ is a $G$-Galois
covering.

A $G$-Galois covering $V \rightarrow W$ is nothing but a
$G$-torsor of $W$, i.e. a principal homogeneous space over $W$
under $G$ \cite[Example 11.3, p.76; Mi1, p.120 and p.43--44]{Mi2}.
If $R \subset S$ are commutative rings, then $S$ is a $G$-Galois
covering of $R$ is equivalent to the fact that $S$ is a Galois
extension of $R$ with group $G$, in the sense of Galois extensions
of commutative rings \cite[Proposition 2.1]{Sw1}. Since $G$ is a
finite group, the assumption of faithful flatness in \cite[p.43;
Mi2, p. 43]{Mi1} guarantees that the morphism is affine and finite
(by the faithfully flat descent \cite[page 20]{Mi1}); when both
$V$ and $W$ are affine schemes, the assumption of faithful
flatness for $V \rightarrow W$ is redundant by \cite[Corollary
2.2]{Sw1}.
\end{defn}

Now we may rephrase Serre's property $\fn{Rat}(G/k)$ as follows.
\begin{defn}\label{d1.5}
Let $k$ be an infinite field and $G$ be a finite group. We say
that the property $\fn{Rat}(G/k)$ holds for the pair ($G, k$), if
there exists a $G$-Galois covering $V \rightarrow W$ where $W$ is
a smooth $k$-rational variety defined over $k$ satisfying the
following condition : For any field $k'$ containing $k$, any
$G$-Galois covering $A$ of $k'$, any nonempty open subset $U
\subset W$, there exists a point $x \in U(k') \subset W$ such that
$Spec(A) \simeq V \times_W Spec(k')$ where the fibre product $V
\times_W Spec(k')$ is defined via the morphism $Spec(k')
\rightarrow \{x \} \subset W$.

Here is an affine version. The property $\fn{Rat}(G/k)$ holds, if
there is a $G$-Galois covering $S$ of $R$ satisfying that (i) $R$
and $S$ are affine $k$-algebra, (ii) $R$ is a localized polynomial
ring, i.e. $R=k[X_1, \cdots, X_n][1/f]$ for some non-zero
polynomial $f$, (iii) for any field $k'$ containing $k$, any
$G$-Galois covering $A$ of $k'$, any $r \in R \setminus \{0 \}$,
there is a $k$-morphism $\phi : R \rightarrow k'$ such that
$\phi(r) \neq 0$ and $A \simeq S \otimes_{\phi} k'$.
\end{defn}

\bigskip
We claim that, if $k$ is an infinite field and $G$ is a finite
group, then ``$k(G)$ is stably $k$-rational" $\Rightarrow$ ``the
property $\fn{Rat}(G/k)$ holds" $\Rightarrow$ ``there is a generic
$G$-Galois extension over $k$".

For the implication``$k(G)$ is stably $k$-rational" $\Rightarrow$
``the property $\fn{Rat}(G/k)$ holds", the same proof of
\cite[Theorem 4.2]{Sw1} works as well in this situation; in
particular, we rely on Kuyk's Lemma, i.e. \cite[Lemma 4.5]{Sw1}.

As to the implication ``the property $\fn{Rat}(G/k)$ holds"
$\Rightarrow$ ``there is a generic $G$-Galois extension over $k$",
suppose that $V \rightarrow W$ is the $G$-Galois covering given in
Definition \ref{d1.5}. Choose an affine open subset $W_0$ of $W$
such that $W_0 \simeq Spec(R)$ for some localized polynomial ring
$R$ (use Lemma \ref{l4.1}, if necessary). Consider the $G$-Galois
covering $V \times_W W_0 \rightarrow W_0$. The fibre product $V
\times_W W_0$ is an affine variety because $V \rightarrow W$ is a
$G$-torsor and $G$ is a finite constant group scheme. Write
$V\times_W W_0 = Spec(S)$. Then the pair ($R,S$) satisfies the
conditions for a generic $G$-Galois extension over $k$ (see
\cite[Definition 1.1]{Sa2} for its definition).

\medskip
By Theorem \ref{t1.2}, we find that, if $k$ is an infinite field
and $G$ is a finite group, then ``$k(G)$ is stably $k$-rational"
$\Rightarrow$ ``the property $\fn{Rat}(G/k)$ holds" $\Rightarrow$
``$k(G)$ is retract $k$-rational". We don't know whether the two
notions ``$\fn{Rat}(G/k)$ holds" and ``$k(G)$ is retract
$k$-rational" are equivalent or not.

\medskip
In \cite{Ku} Kunyavskii studies the birational classification of
3-dimensional algebraic tori over a field $k$. He gives a list of
all those tori which are $k$-rational; the remaining ones are not
stably $k$-rational. In a private communication during 2009
Kunyavskii informed me that, from the proof in \cite{Ku}, it is
not difficult to deduce that a 3-dimensional algebraic torus over
$k$ is not retract $k$-rational if and only if it is not stably
$k$-rational.

\bigskip
We organize this paper as follows. We review basic notions of
multiplicative group actions in Section 2. In Section 3 Saltman's
work on retract rationality is reviewed. The transitivity theorem
of retract rationality is proved in Section 4. Applications are
given in Section 5 where Theorem \ref{t5.10} is the main result.
In Section 6, we study the fixed subfields of monomial actions;
Theorem \ref{t6.6} is the generalization of Barge's Theorem.

\begin{idef}{Standing notations.}
In discussing retract rationality, we always assume that the
ground field is infinite (see Definition \ref{d1.1}). Thus,
throughout this paper, we will assume that $k$ is an infinite
field, unless otherwise specified. A finitely generated field
extension $L$ of $k$ is called a $k$-field for short.
$k(x_1,\ldots,x_n)$ or $k(X_1,\ldots,X_n)$ denotes the rational
function field of $n$ variables over $k$. For emphasis, recall
$k(G)=k(x_g:g\in G)^G$.

We denote by $\zeta_n$ a primitive $n$-th root of unity in some
extension field of $k$. When we write $\zeta_n\in k$, it is
understood that $\fn{char}k=0$ or $\fn{char}k=p>0$ with $p\nmid
n$. Similarly, when we write $\fn{char}k\nmid n$, it is understood
that $\fn{char}k=0$ or $\fn{char}k=p>0$ with $p\nmid n$.

For brevity, we will called $k[X_1,\ldots,X_n][1/f]$ a localized
polynomial ring when $k[X_1,\ldots,X_n]$ is a polynomial ring and
$f\in k[X_1,\ldots,X_n]\backslash \{0\}$ (see the definition of
retract rationality in Definition \ref{d1.1}). An affine domain
over $k$ or an affine $k$-domain (or simply an affine domain) is
an integral domain of the form $k[\alpha_1,\ldots,\alpha_m]$ for
finitely many elements $\alpha_1,\ldots,\alpha_m$.

All the groups in this article are finite groups.
$C_n$ denotes the cyclic group of order $n$.
$\bm{Z}[\pi]$ is the group ring of the finite group $\pi$ over $\bm{Z}$.
The exponent of a group $G$ is the least common multiple of the orders of elements in $G$.
\end{idef}

\section{Multiplicative group actions}

Let $\pi$ be a finite group.
A $\pi$-lattice $M$ is a finitely generated $\bm{Z}[\pi]$-module such that $M$ is a free abelian group
when it is regarded as an abelian group.

For a $\pi$-lattice $M$, $k[M]$ denotes the Laurent polynomial
ring and $k(M)$ is the quotient field of $k[M]$. Explicitly, if
$M=\bigoplus_{1\le i\le m}\bm{Z}\cdot x_i$\vspace*{2pt} as a free
abelian group, then $k[M]=k[x_1^{\pm 1},\ldots,x_m^{\pm 1}]$ and
$k(M)=k(x_1,\ldots,x_m)$. Since $\pi$ acts on $M$, it will act on
$k[M]$ and $k(M)$ by $k$-automorphisms, i.e.\ if $\sigma\in \pi$
and \vspace*{2pt}$\sigma\cdot x_j=\sum_{1\le i\le m}a_{ij}x_i\in
M$, then we define the multiplicative action of $\sigma$ on $k[M]$
and $k(M)$ by $\sigma\cdot x_j=\prod_{1\le i\le m} x_i^{a_{ij}}$.
\vspace*{2pt}

The multiplicative action of $\pi$ on $k(M)$ is called a purely
monomial action in \cite{HK1}. If $\pi$ is a group acting on the
rational function field $k(x_1,\ldots,x_m)$ by $k$-automorphism
such that $\sigma\cdot x_j=c_j(\sigma)\cdot \prod_{1\le i\le
m}x_i^{a_{ij}}$\vspace*{2pt} where $\sigma \in \pi$,
$a_{ij}\in\bm{Z}$ and $c_j(\sigma)\in k\backslash \{0\}$, such a
multiplicative group action is called a monomial action.

\begin{defn} \label{d2.1}
Let $M=\bigoplus_{1\le j\le m} \bm{Z}\cdot x_j$ be a $\pi$-lattice
and $\pi$ act on $k(M)=k(x_1,\ldots,$ $x_m)$ by purely monomial
$k$-automorphisms. The fixed field, denoted by $k(M)^\pi$, is
defined as $k(M)^\pi=\{f\in k(x_1,\ldots,x_m):\sigma\cdot f=f$ for
any $\sigma\in\pi\}$. This field $k(M)^\pi$ was designated as
$k(M,\pi)$ by Saltman in \cite{Sa5}.

On the other hand, the fixed field for a monomial action is
denoted by $k_\alpha(M)^\pi$ (here $\alpha$ designates the
extension of $\bm{Z}[\pi]$-modules associated to the monomial
action, which will be defined below). Precisely, if $\pi$ acts on
$k(M)=k(x_1,\ldots,x_m)$ by monomial $k$-automorphisms, define
$M_\alpha$ to be the (multiplicatively written)
$\bm{Z}[\pi]$-module generated by $x_1,\ldots,x_m$ and
$k^{\times}$(:$= k\backslash \{0\}$) in
$k(x_1,\ldots,x_m)\backslash\{0\}$. Thus we obtain a short exact
sequence of $\bm{Z}[\pi]$-modules $0\to k^{\times}\to M_\alpha \to
M\to 0$; label this short exact sequence (or the module extension)
as $\alpha$. Define $k_\alpha(M)^\pi=\{f\in
k(x_1,\ldots,x_m):\sigma\cdot f=f$ for any $\sigma\in\pi\}$.

 Note that $k_\alpha(M)^\pi$ of this article agrees
with the notation of Saltman in \cite[p.538]{Sa6}; our notation
$k_\alpha(M)^\pi$ also agrees with Saltman's notation in
\cite[p.535]{Sa7}, except that, $M_\alpha$ in \cite{Sa7} is the
multiplicative subgroup generated by $x_1,\ldots,x_m$ and $\mu$
where $k$ is assumed to be algebraically closed and $\mu$ denotes
the group of all roots of unity in $k^{\times}$.
\end{defn}

\begin{defn} \label{d2.2}
Let $K$ be a $k$-field, $\pi$ be a finite group, and $M
=\bigoplus_{1\le j\le m}\bm{Z}\cdot x_j$ be a $\pi$-lattice.
Suppose that $\pi$ acts on $K$ by $k$-automorphisms of $K$ and
$\pi$ acts on $K(M)$ by monomial $k$-automorphism, i.e.\
$\sigma\cdot x_j=c_j(\sigma)\cdot \prod_{1\le i\le m}x_i^{a_{ij}}$
where $\sigma\in\pi$,\vspace*{2pt} $c_j(\sigma)\in
K\backslash\{0\}$, $a_{ij}\in\bm{Z}$. We will denote the fixed
field by $K_\alpha(M)^\pi$ where $\alpha: 0\to K^{\times}\to
M_\alpha \to M\to 0$ is the associated extension of this monomial
action of $\pi$. If $\pi$ acts on $K(M)$ by purely monomial
automorphisms, we will write $K(M)^\pi$ for $K_\alpha(M)^\pi$.

Note that it is not necessary to assume that the action of $\pi$
on the $k$-field $K$ is faithful. In case $\pi$ acts faithfully on
$K$ and acts on $K(M)$ by purely monomial $k$-automorphisms, then
$K(M)^\pi$ is just the function field of some algebraic torus
defined over $K^\pi$ split by $K$ and with character group $M$
(see \cite{Vo2}).
\end{defn}

We recall some basic facts of the theory of flabby (flasque)
$\pi$-lattices developed by Voskresenskii, Endo and Miyata,
Colliot-Th\'el\`ene and Sansuc, etc. \cite{Vo2,CTS1}. We refer the
reader to \cite{Sw1,Sw2,Lo} for a quick review of the theory.

\begin{defn} \label{d2.3}
A $\pi$-lattice $M$ is called a permutation lattice if $M$ has a
$\bm{Z}$-basis permuted by $\pi$. $M$ is called an invertible (or
permutation projective) lattice, if it is a direct summand of some
permutation lattice. A $\pi$-lattice $M$ is called a flabby (or
flasque) lattice if $H^{-1}(\pi',M)=0$ for any subgroup $\pi'$ of
$\pi$ (note that all the cohomology groups in this paper, in
particular $H^{-1}(\pi',M)$, are the Tate cohomology groups).
Similarly, $M$ is called coflabby if $H^1(\pi',M)=0$ for any
subgroup $\pi'$ of $\pi$. More generally, if $N$ is a
$\bm{Z}[\pi]$-module, we will say that $N$ is $H^1$ trivial if
$H^1(\pi',N)=0$ for any subgroup $\pi'$ of $\pi$.

Let $\mathcal{L}_{\pi}$ be the set of all $\pi$-lattices. We
define a similarity relation on $\mathcal{L}_{\pi}$: If
$M_1,M_2\in \mathcal{L}_{\pi}$, then $M_1\sim M_2$ if and only if
$M_1\oplus Q_1\simeq M_2\oplus Q_2$ for some permutation lattices
$Q_1$ and $Q_2$. The set of all similarity classes is denoted by
$\mathcal{L}_{\pi}/{\sim}$; $[M]$ denotes the similarity class
containing $M$ in $\mathcal{L}_{\pi}/{\sim}$. Note that the
operation of the direct sum in $\mathcal{L}_{\pi}$ induces a
commutative monoid structure on $\mathcal{L}_{\pi}/{\sim}$.
\end{defn}

\begin{lemma}[{\cite[Lemma 8.4; Le, Proposition 1.2]{Sw1}}] \label{l2.4}

{\rm (1)} If $E$ is an invertible $\pi$-lattice, then $E$ is
flabby and coflabby.

{\rm (2)} If $E$ is an invertible $\pi$-lattice and $C$ is a
coflabby $\pi$-lattice, then any short exact sequence $0\to C\to
N\to E\to 0$ splits.
\end{lemma}

\begin{theorem}[Endo and Miyata {\cite[Theorem 3.4; Lo, 2.10.1]{Sw2}}] \label{t2.5}
Let $\pi$ be a finite group. Then all the flabby $\pi$-lattices
are invertible if and only if all the Sylow subgroups of $\pi$ are
cyclic.
\end{theorem}

\begin{theorem}[Colliot-Th\'el\`ene and Sansuc {\cite[Lemma 8.5; Lo, Lemma 2.6.1]{Sw1}}] \label{t2.6}
For any $\pi$-lattice $M$,
there is a short exact sequence of $\pi$-lattices $0\to M\to P\to F\to 0$ where $P$ is a permutation lattice
and $F$ is a flabby lattice.
\end{theorem}

\begin{defn} \label{d2.7}
The exact sequence $0\to M\to P\to F\to 0$ in the above theorem is
called a flabby resolution of the $\pi$-lattice $M$. The flabby
class of $M$, denoted by $[M]^{fl}$, is defined as $[M]^{fl}=[F]
\in \mathcal{L}_{\pi}/{\sim}$. Note that $[M]^{fl}$ is
well-defined : If $[M]=[M']$, $[M]^{fl}=[F]$ and $[M']^{fl}=[F']$,
then $F\oplus Q\simeq F'\oplus Q'$ for some permutation lattices
$Q$ and $Q'$, and therefore $[F]=[F']$ (see, for example,
\cite[Lemma 8.7]{Sw1}).

When we say that $[M]^{fl}$ is invertible, we mean that
$[M]^{fl}=[E]$ for some invertible lattice $E$.
\end{defn}

\begin{theorem}[Saltman {\cite[Theorem 3.14]{Sa4}}] \label{t2.8}
Let $K$ be a finite Galois field extension of $k$ with
$\pi=Gal(K/k)$. For any $\pi$-lattice $M$, $K(M)^\pi$ is retract
$k$-rational if and only if $[M]^{fl}$ is invertible.
\end{theorem}

As an application of Theorem \ref{t2.8}, we prove the following theorem.

\begin{theorem}[Voskresenskii \cite{Vo1}] \label{t2.9}
Let $k$ be an infinite field with $\fn{char}k \ne 2$. If
$k(\zeta_{2^n})$ is not a cyclic extension of $k$, then
$k(C_{2^n})$ is not retract $k$-rational. Thus $k(C_{2^n})$ is not
rational over $k$.
\end{theorem}

\begin{proof}
Write $C_{2^n}=\langle \sigma\rangle$ and $V=\bigoplus_{0\le i\le
2^n-1} k\cdot x(\sigma^i)$\vspace*{2pt} be the regular
representation space of $C_{2^n}$. Then
$k(C_{2^n})=k(x(\sigma^i):0\le i\le
2^n-1)^{\langle\sigma\rangle}$.

Let $\zeta=\zeta_{2^n}$ and $\pi=\fn{Gal}(k(\zeta)/k)$.
Extend the actions of $\sigma$ and $\pi$ to $k(\zeta)(x(\sigma^i):0\le i\le 2^n-1)$ so that $\pi$
acts trivially on $x(\sigma^i)$ and $\sigma$ acts trivially on $k(\zeta)$.
For $0\le i\le 2^n-1$, define
\[
y_i=\sum_{0\le j\le 2^n-1}\zeta^{-ij}\cdot x(\sigma^j)\in \bigoplus_{0\le j\le 2^n-1} k(\zeta)\cdot x(\sigma^j).
\]

It follows that $k(C_{2^n})=k(x(\sigma^i):0\le i\le 2^n-1)^{\langle\sigma\rangle} %
=\{k(\zeta)(x(\sigma^i):0\le i\le 2^n-1)^\pi\}^{\langle \sigma\rangle} %
=k(\zeta)(y_i:0\le i\le 2^n-1)^{\langle \sigma,\pi\rangle}$.

It is easy to see that $\sigma\cdot y_i=\zeta^i y_i$ for $0\le
i\le 2^n-1$. Moreover, if $\tau_t\in \pi$ is defined by
$\tau_t(\zeta)=\zeta^t$, then $\tau_t(y_i)=y_{ti}$ for $0\le i\le
2^n-1$ (note that the subscript $ti$ of $y_{ti}$ is taken modulo
$2^n$). It follows that $k(\zeta)(y_i: 0\le i\le 2^n-1)^{\langle
\sigma,\pi\rangle}$ $=k(\zeta)(y_i: 1\le i\le 2^n-1)^{\langle
\sigma,\pi\rangle}(y_0)$.

Let $N$ be the multiplicative subgroup of $k(\zeta)(y_i:1\le i\le
2^n-1)\backslash \{0\}$ generated by $y_1, y_2,\ldots,y_{2^n-1}$.
Since $\pi$ acts on $N=\langle y_i: 1\le i\le 2^n-1\rangle$, $N$
is a $\pi$-lattice. Similarly, $\pi$ acts on $\langle \zeta\rangle
\simeq \bm{Z}/2^n\bm{Z}$; thus we may regard $\bm{Z}/2^n\bm{Z}$ as
a finite $\bm{Z}[\pi]$-module (note that $\tau_t \cdot
\bar{i}=\bar{ti}$ for any $\bar{i} \in \bm{Z}/2^n\bm{Z}$). Define
a $\pi$-morphism $\Phi$ by
\[
\Phi: N\to \bm{Z}/2^n\bm{Z}
\]
where, for any monomial $y=\prod_{1\le j\le 2^n-1}
y_j^{\lambda_j}$\vspace*{2pt} with $\lambda_j\in \bm{Z}$, define
$\Phi(y)=\sigma(y)/y$ (note that $\sigma(y)/y \in \langle
\zeta\rangle$, and thus can be regarded as an element of
$\bm{Z}/2^n\bm{Z}$).

Define $M=\fn{Ker}(\Phi)$, which is a $\pi$-lattice. It follows
that $k(\zeta)(y_i: 1\le i\le 2^n-1)^{\langle \sigma,\pi\rangle}$
$=\{k(\zeta)(y_i:1\le i\le 2^n-1)^{\langle
\sigma\rangle}\}^\pi=k(\zeta)(M)^\pi$.

We compare the above construction with that in \cite[p.310]{Le}.
It is clear that $N\simeq \bm{Z}^{C(q)}$ where $q=2^n$ and
$\bm{Z}^{C(q)}$ is Lenstra's notation. Thus $M\simeq I_q$ in
Lenstra's notation. By \cite[Proposition 3.1 and Propositiion
3.2]{Le}, $H^1(\pi',I_q)=0$ for any subgroup $\pi'$ of $\pi$ and
$H^{-1}(\pi_0,I_q)\simeq \bm{Z}/2\bm{Z}$ if $\pi_0$ is the unique
subgroup of $\pi$ isomorphic to $C_2\times C_2$ (also see
\cite[p.79]{Vo2}). Thus $I_q (\simeq M)$ is coflabby, but not
flabby.

Let $0\to M\to P\to F\to 0$ by any flabby resolution of $M$ so
that $P$ is permutation and $F$ is flabby. Suppose that $F$ is
invertible. By Lemma \ref{l2.4}, this exact sequence splits. Thus
$P\simeq M\oplus F$. In particular, $M$ is invertible. By Lemma
\ref{l2.4} again, $M$ is flabby. This leads to a contradiction to
the previous assertion that $M$ is not flabby.

Apply Theorem \ref{t2.8}.
We find that $k(\zeta)(M)^\pi$ ($=k(C_{2^n})$) is not retract $k$-rational.
\end{proof}

\begin{coro} \label{c2.10}
Let $k$ be a field such that $k(\zeta_{2^n})$ is not cyclic over
$k$. Let $G=H \rtimes C_{2^n}$ where $H$ is a normal subgroup of
the finite group $G$ with $C_{2^n}$ acting on it. Then $k(G)$ is
not retract $k$-rational.
\end{coro}

\begin{proof}
Suppose $k(G)$ is retract $k$-rational. By Theorem \ref{t3.5},
$k(C_{2^n})$ is retract $k$-rational, which contradicts to Theorem
\ref{t2.9}.
\end{proof}

\begin{remark}
Let $M$ be the $\pi$-lattice defined in the proof of Theorem
\ref{t2.9}. Voskresenskii showed that $[M]^{fl}$, the flabby class
of $M$, is not invertible \cite[p.97--99; Vo2, p.79]{Vo1}; the
same result was obtained by Lenstra \cite[p.310--311]{Le}. Saltman
showed that $\bm{Q}(G)$ is not retract $\bm{Q}$-rational if $G$ is
a finite abelian group containing an element of order $2^n$ with
$n\ge 3$ by using Wang's counter-example to Grunwald Theorem
\cite[Theorem 5.11]{Sa2}. Sonn generalized Saltman's Theorem and
proved that, $\bm{Q}(G)$ is not retract $\bm{Q}$-rational if $G$
is any finite group containing a normal subgroup $H$ such that
$G/H \simeq C_{2^n}$ with $n\ge 3$ \cite{So}.
\end{remark}

\section{Criteria of retract rationality}

In this section we recall several results about retract rationality,
which will be used subsequently.

First we define the unramified Brauer group of a $k$-field $L$.

\begin{defn}[{\cite[p.226]{Sa3,Sa5}}] \label{d3.1}
Let $L$ be a $k$-field.
Define the unramified Brauer group of $L$ over $k$,
denoted by $\fn{Br}_{v,k}(L)$, as $\fn{Br}_{v,k}(L)=\bigcap_R \fn{Br}(R)\subset \fn{Br}(L)$
where $R$ runs over all discrete $k$-valuation rings whose quotient fields are equal to $L$,
and $\fn{Br}(R)$ denotes the Brauer group of $R$.
See \cite[Section 3; Sa7, Theorem 12]{Bo} for more results about unramified Brauer groups.
\end{defn}

\begin{theorem}[Saltman {\cite[Section 2]{Sa5}}] \label{t3.2}
{\rm (i)} Let $L$ be a $k$-field. If $L$ is retract $k$-rational,
then $\fn{Br}_{v,k}(L)\simeq \fn{Br}(k)$. In particular, when $k$
is algebraically closed and $L$ is retract $k$-rational, then
$\fn{Br}_{v,k}(L)=0$.

{\rm (ii)} If $K\subset L$ are $k$-fields and $L$ is retract $K$-rational,
then $\fn{Br}_{v,k}(K)\simeq \fn{Br}_{v,k}(L)$.
\end{theorem}

\medskip
Note that $\fn{Br}_{v, \bm{C}}(L)=0$ is just a necessary condition
for a $\bm{C}$-field $L$ to be retract $\bm{C}$-rational. It is
not a sufficient condition. In fact, Peyre shows that, there is a
group $G$ of order $p^{12}$ such that $\bm{C}(G)$ is not retract
$\bm{C}$-rational but $\fn{Br}_{v, \bm{C}}(\bm{C}(G))=0$
\cite{Pe}.

\begin{defn}[\cite{Sa5}] \label{d3.3}
Let $K\subset L$ be $k$-field. $K$ is called a dense retraction of
$L$ if there is a regular affine $K$-algebra $R$ such that (i) the
quotient field of $R$ is $L$, and (ii) for any $r\in
R\backslash\{0\}$, there is a $K$-algebra morphism
$\varphi:R[1/r]\to K$.

We will prove in Lemma \ref{l5.2} that, if $L$ is retract
$k$-rational, then $k$ is a dense retraction of $L$.
\end{defn}

Now consider retract rationality. We reformulate Saltman's results
of \cite{Sa2} in terms of retract rationality by applying Theorem
\ref{t1.2}.

\begin{lemma} \label{l3.4}
{\rm (i) (\cite[Proposition 3.6]{Sa4})} Let $L$ be a $k$-field,
$L(x_1,\ldots, x_n)$ be the rational function field over $L$. Then
$L$ is retract $k$-rational if and only if so is
$L(x_1,\ldots,x_n)$.

{\rm (ii) (\cite[Theorem 1.5 and Theorem 3.1]{Sa2})} Let $G=G_1
\times G_2$. Then $k(G)$ is retract $k$-rational if and only if so
are $k(G_1)$ and $k(G_2)$.

{\rm (iii) (\cite[Lemma 1.1]{Sa5})} Let $K\subset L$ be
$k$-fields. If $L$ is retract $k$-rational and $K$ is a dense
retraction of $L$, then $K$ is retract $k$-rational.

{\rm (iv) (\cite[Theorem 1.3]{Sa5})} Let $K$ be a finite Galois
field extension of $k$ with $\pi=Gal(K/k)$, and $M$ be any
$\pi$-lattice. Then $k$ is a dense retraction of $K(M)^\pi$.
\end{lemma}

\begin{theorem}[{\cite[Theorem 3.1 and Theorem 3.5]{Sa2}}] \label{t3.5}
Let $G=N\rtimes G_0$ where $N$ is a normal subgroup of $G$ with
$G_0$ acting on $N$.

{\rm (1)} If $k(G)$ is retract $k$-rational, so is $k(G_0)$.

{\rm (2)} Assume furthermore that $N$ is abelian and $\gcd\{|N|,|G_0|\}=1$.
If both $k(N)$ and $k(G_0)$ are retract $k$-rational, so is $k(G)$.
\end{theorem}

\begin{remark}
For more results about sufficient conditions to ensure that
$k(N\rtimes G_0)$ is retract $k$-rational, see \cite[Theorems
1.11, 1.12, 1.13 and Theorem 4.3]{Ka2}.
\end{remark}

We recall a reduction theorem for Noether's problem.

\begin{theorem}[{\cite[Theorem 1.1]{KP}}] \label{t3.6}
Let $k$ be a field with $\fn{char}k=p>0$ and $\widetilde{G}$ be a
group extension defined by $1\to C_p\to\widetilde{G}\to G\to 1$.
Then $k(\widetilde{G})$ is rational over $k(G)$.
\end{theorem}

\begin{theorem}[Saltman {\cite[Theorem 4.12]{Sa4}}] \label{t3.7}
Let $k$ be an infinite field and $G$ be a finite abelian group of
exponent $e=2^rm$ with $2\nmid m$. Then $k(G)$ is retract
$k$-rational if and only if either $\fn{char}k=2$, or
$k(\zeta_{2^r})$ is a cyclic extension over $k$.
\end{theorem}

\begin{proof}
If $\fn{char}k=p>0$ and $p\mid e$, choose an element $g\in G$ of
order $p$. Consider $1\to \langle g\rangle \to G\to G/\langle
g\rangle \to 1$ and apply Theorem \ref{t3.6}. Since $k(G)$ is
rational over $k(G/\langle g\rangle)$, it follows that $k(G)$ is
retract $k$-rational if and only if so is $k(G/\langle g\rangle)$
by Lemma \ref{l3.4}. Thus we may assume that
$\gcd\{\fn{char}k,|G|\}=1$ without loss of generality.

Write $G \simeq \prod_q C_q$ where these $q$'s are some prime
powers with $\gcd\{\fn{char}k, q \}=1$. By Lemma \ref{l3.4}, it
suffices to check whether each $k(C_q)$ is or is not retract
$k$-rational.

If $\fn{char}k=2$, then $q$ is an odd integer by the above
assumption. Thus $k(C_q)$ is retract $k$-rational by \cite[Theorem
2.1]{Sa2}. From now on, we assume that $\fn{char}k \neq 2$.

By \cite[Theorem 2.1]{Sa2}, $k(C_q)$ is retract $k$-rational if
$q$ is odd or $q$ is even with $k(\zeta_q)$ being cyclic over $k$.
When $k(\zeta_q)$ is not cyclic over $k$, then $k(C_q)$ is not
retract $k$-rational by Theorem \ref{t2.9}.
\end{proof}

\begin{remark}
Voskresenskii shows that, if $G=C_{2^r}$, then $k(G)$ is
$k$-rational $\Leftrightarrow k(G)$ is retract $k$-rational
$\Leftrightarrow$ either $\fn{char}k=2$ or $k(\zeta_{2^r})$ is
cyclic over $k$ \cite[p.79]{Vo2}. For any odd prime number $p$,
$\bm{Q}(C_p)$ is always retract $\bm{Q}$-rational by the above
theorem, while $\bm{Q}(C_{47})$ is not $\bm{Q}$-rational by Swan
(see, for example, \cite[p. 299]{Le}), and thus not stably
$\bm{Q}$-rational by \cite[Proposition 5.6]{Le}.
\end{remark}

Here is another criterion for retract rationality.
\begin{example}[{\cite[p.2763]{Ka2}}] \label{ex3.8}
Let $k$ be any infinite field, $G$ be a non-abelian $p$-group of
exponent $p$ and of order $p^3$ or $p^4$. Then $k(G)$ is retract
$k$-rational.
\end{example}

\section{A transitivity theorem}

Before proving the transitivity theorem, we recall a lemma due to
Swan.

\begin{lemma}[{\cite[Lemma 4.3]{Sw1}}] \label{l4.1}
Let $L$ be a $k$-field,
$R_1$ and $R_2$ be affine $k$-domains contained in $L$ such that the quotient fields of $R_1$ and $R_2$
are equal to $L$.
Then there are $r_1\in R_1\backslash\{0\}$, $r_2\in R_2\backslash\{0\}$ such that
$R_1[1/r_1]=R_2[1/r_2]$.
\end{lemma}

\begin{theorem} \label{t4.2}
Let $K\subset L$ be $k$-fields.
If $K$ is retract $k$-rational and $L$ is retract $K$-rational,
then $L$ is retract $k$-rational.
\end{theorem}

\begin{proof}
Geometrically this result looks clear. Here is a rigorous proof.

Step 1. By assumptions, there exist an affine $K$-domain $B$, an
affine $k$-domain $S$, localized polynomial ring
$K[X_1,\ldots,X_n][1/f]$, $k[Y_1,\ldots,Y_m][1/g]$, and
$K$-algebra morphisms $\varphi:B\to K[X_1,\ldots,X_n][1/f]$,
$\psi:K[X_1,\ldots,X_n][1/f]\to B$, $k$-algebra morphisms
$\varphi_1:S\to k[Y_1,\ldots, Y_m][1/g]$,
$\psi_1:k[Y_1,\ldots,Y_m][1/g]\to S$ satisfying that

(i) the quotient fields of $B$ and $S$ are $L$ and $K$ respectively, and

(ii) $\psi\circ\varphi =1_B$, $\psi_1\circ \varphi_1=1_S$.

We will find a subring $A$ of $L$ and an affine $k$-domain $R$ of $K$ such that

(i) $A=R[\alpha_1,\ldots,\alpha_t]$ for some $\alpha_1,\ldots,\alpha_t\in L$,

(ii) the quotient fields of $A$ and $R$ are $L$ and $K$ respectively, and

(iii) the above morphisms $\varphi$, $\psi$, $\varphi_1$, $\psi_1$ are ``well-defined" for $A$ and $R$,
i.e.\ the ``natural extensions" of these morphisms (still denoted by $\varphi$, $\psi$, $\varphi_1$, $\psi_1$,
by abusing the notations) $\varphi:A\to R[X_1,\ldots,X_n][1/f]$,
$\psi:R[X_1,\ldots,X_n][1/f]\to A$, $\varphi_1:R\to k[Y_1,\ldots,Y_m][1/g_0]$,
$\psi_1:k[Y_1,\ldots,Y_m][1/g_0]\to R$ are well-defined (where $g_0=gg_1$ for some non-zero polynomial $g_1$)
and satisfy $\psi\circ\varphi =1_A$, $\psi_1\circ\varphi_1=1_R$.

The above assertion seems obvious in some sense, although a formal
proof is tedious. We provide the proof in the following.

Note that in choosing the localized polynomials $K[X_1,\ldots,X_n][1/f]$ and $k[Y_1,\ldots,\break Y_m][1/g]$,
we may assume that $X_1,\ldots,X_n,Y_1,\ldots,Y_m$ are algebraically independent over $K$.
In fact, these subrings may be chosen from the rational function field $K(X_1,\ldots,X_n,\break Y_1,\ldots,Y_m)$.

Write $B=K[\alpha_1,\alpha_2,\ldots,\alpha_t]$ for some $\alpha_1,\ldots,\alpha_t\in L$.
Let $R_1$ be an affine $k$-domain whose quotient field is $K$.
Thus the quotient field of $R_1[\alpha_1,\ldots,\alpha_t]$ is $L$.

We will enlarge $R_1$ by adjoining additional elements of $K$ to $R_1$.
First $f\in K[X_1,\ldots,X_n]$.
Adjoin all the coefficients of $f$ into $R_1$.
Then consider $\varphi(\alpha_j)$ for $1\le j\le t$.
Since $\varphi(\alpha_j)=f_j/f^l$ for some $f_j\in K[X_1,\ldots,X_n]$.
Adjoin all the coefficients of all these $f_j$ to $R_1$ also.
Call this new affine $k$-domain $R_2$.
It follows that $f\in R_2[X_1,\ldots,X_n]$ and $\varphi:R_2[\alpha_1,\ldots,\alpha_t]\to R_2[X_1,\ldots,X_n][1/f]$
is well-defined.

Now consider $\psi(X_1),\ldots,\psi(X_n)$ and $\psi(1/f)$.
They lie in $B=K[\alpha_1,\ldots,\alpha_t]$.
Thus they belong to the subring $R_2[\alpha_1,\ldots,\alpha_t][1/\beta]$ for a fixed element $\beta\in K\backslash\{0\}$.
Adjoin $1/\beta$ to $R_2$.
Call this affine $k$-domain $R_3$.
We conclude that the $R_3$-algebra morphisms $\varphi:R_3[\alpha_1,\ldots,\alpha_t]\to R_3[X_1,\ldots,X_n][1/f]$,
$\psi:R_3[X_1,\ldots,X_n][1/f]\to R_3[\alpha_1,\ldots,\alpha_t]$ are well-defined and satisfy $\psi\circ\varphi=1$.

Consider the affine $k$-domain $S$.
Apply Lemma \ref{l4.1}.
We find $r\in R_3\backslash\{0\}$ and $r_1\in S\backslash\{0\}$ so that $R_3[1/r]=S[1/r_1]$.
Define $R=R_3[1/r]$ and $A=R[\alpha_1,\ldots,\alpha_t]$.

Note that $\varphi_1(r_1)=g_1/g^{l'}$ for some non-zero polynomial
$g_1\in k[Y_1,\ldots,Y_m]$. Define $g_0=g\cdot g_1$. Then
$\varphi_1:S[1/r_1]\to k[Y_1,\ldots,Y_m][1/g_0]$ is well-defined.
It is not difficult to check that, in the morphism
$\psi_1:k[Y_1,\ldots,Y_m][1/g]\to S$, the element $\psi_1(g)$ is a
unit in $S$. Thus $\psi_2(g_1)=r_1u$ for some unit $u\in S$. It
follows that $\psi_1: k[Y_1,\ldots,Y_m][1/g_0]\to S[1/r_1]$ is
also well-defined. Thus, the $k$-algebra morphisms $\varphi_1:R\to
k[Y_1,\ldots,Y_m][1/g_0]$ and $\psi_2:k[Y_1,\ldots,Y_m][1/g_0]\to
R$ satisfying $\psi_1\circ \varphi_1 =1_R$. So are the $R$-algebra
morphisms $\varphi:A\to R[X_1,\ldots,X_n][1/f]$ and
$\psi:R[X_1,\ldots,X_n][1/f]\to A$. Done.

\bigskip
Step 2.
Let $C_0:=R[X_1,\ldots,X_n][1/f]$.
Then we have $R$-algebra morphisms $\varphi:A\to C_0$ and $\psi:C_0\to A$ with $\psi\circ\varphi=1_A$.
Note that $A=R[\alpha_1,\ldots,\alpha_t]$ is an affine $k$-domain whose quotient field is $L$.
We will define a localized polynomial $C$ related to $A$ and $C_0$.

Since $f\in R[X_1,\ldots,X_n]$,
write $f=\sum_\lambda a_\lambda\cdot X^\lambda$ where $X^\lambda=X_1^{\lambda_1}X_2^{\lambda_2}\cdots X_n^{\lambda_n}$
and $a_\lambda\in R$.
Write $\varphi_1(a_\lambda)=b_\lambda/g_0^N$ for all $\lambda$ where $b_\lambda\in k[Y_1,\ldots,Y_m]$.
Define $f_0$ and $h$ by
\begin{align*}
f_0 &= \left.\left(\sum_\lambda b_\lambda X^\lambda\right)\right/ g_0^N, \\
h &= \sum_\lambda b_\lambda X^\lambda\in k[X_1,\ldots,X_n,Y_1,\ldots,Y_m]\backslash \{0\}.
\end{align*}

Define $C=k[X_1,\ldots,X_n,Y_1,\ldots,Y_m][1/(g_0h)]$.

From the $k$-algebra morphism $R\stackrel{\varphi_1}{\longrightarrow} k[Y_1,\ldots,Y_m][1/g_0] %
\stackrel{\psi_1}{\longrightarrow}R$,
extend the base to $k[X_1,\ldots,X_n]$,
i.e.\ define $k$-algebra morphisms $\varphi_2:R[X_1,\ldots,X_n]\to k[X_1,\ldots,X_n,\break Y_1,\ldots,Y_m][1/g_0]$
and $\psi_2:k[X_1,\ldots,X_n,Y_1,\ldots,Y_m][1/g_0]\to R[X_1,\ldots,X_n]$ by requiring that both $\varphi_2$
and $\psi_2$ are morphisms over $k[X_1,\ldots,X_n]$ and define $\varphi_2(r)=\varphi_1(r)$ for any $r\in R$,
$\psi_2(G)=\psi_1(G)$ for any $G\in k[Y_1,\ldots,Y_m][1/g_0]$.

Note that $f\in R[X_1,\ldots,X_n]$ and $\varphi_2(f)=f_0=h/g_0^N$ by the above definition.
Hence $\varphi_2:C_0=R[X_1,\ldots,X_n][1/f]\to C=k[X_1,\ldots,X_n,Y_1,\ldots,Y_m][1/(g_0h)]$ is well-defined.

Moreover, from the relation $b_\lambda=g_0^N\cdot \varphi_1(a_\lambda)$,
we get $\psi_1(b_\lambda)=\psi_1(g_0)^N\cdot(\psi_1\circ\varphi_1)(a_\lambda) \break =a_\lambda\cdot\psi_1(g_0)^N$.
Note that $\psi_1(g_0)$ is a unit in $R$.
It follows that $\psi_2(h)=\psi_2(\sum_\lambda b_\lambda X^\lambda)$
$=\sum_\lambda \psi_2(b_\lambda)X^\lambda %
=\sum_\lambda \psi_1(b_\lambda) X^\lambda=\psi_1(g_0)^N\cdot \sum_\lambda a_\lambda X^\lambda %
=\psi_1(g_0)^N\cdot f$ is also a unit in $C_0$ since $1/f\in C_0$.
Thus $\psi_2:C=k_0[X_1,\ldots,X_n,Y_1,\ldots,Y_m][1/(g_0h)]\to C_0=R[X_1,\ldots,X_n][1/f]$ is also well-defined.
Clearly $\psi_2\circ \varphi_2=1_{C_0}$.

\bigskip
Step 3.
Note that we have the following diagram
\[
\begin{CD}
A @>\varphi>> C_0 @>\psi>> A \\
@.  @V{\varphi_2}VV \\
@.  C \\
@.  @V{\psi_2}VV \\
A @>\varphi>> C_0 @>\psi>> A
\end{CD}
\]
define $\widetilde{\varphi}=\varphi_2\circ \varphi$ and $\tilde{\psi}=\psi\circ\psi_2$.
It follows that $\tilde{\psi}\cdot \tilde{\varphi}=1_A$.
Thus $L$ is retract $k$-rational.
\end{proof}

\section{Applications}

We recall a known result which will be used subsequently.

\begin{theorem}[{\cite[Theorem 1]{HK3}}] \label{t5.1}
Let $L$ be any field and $G$ be a finite group acting on
$L(x_1,\ldots,x_m)$, the rational function field of $m$ variables
over a field $L$. Suppose that

{\rm (i)} for any $\sigma\in G$, $\sigma(L)\subset L$;

{\rm (ii)} The restriction of the action of $G$ to $L$ is faithful;

{\rm (iii)} for any $\sigma\in G$,
\[
\begin{pmatrix} \sigma(x_1) \\ \vdots \\ \sigma(x_m) \end{pmatrix}
=A(\sigma)\cdot \begin{pmatrix} x_1 \\ \vdots \\ x_n \end{pmatrix}+B(\sigma)
\]
where $A(\sigma)\in GL_m(L)$ and $B(\sigma)$ is an $m\times 1$ matrix over $L$.

Then there exist $z_1,\ldots,z_m\in L(x_1,\ldots,x_m)$ such that $L(x_1,\ldots,x_m)=L(z_1,\ldots,z_m)$
and $\sigma(z_j)=z_j$ for any $\sigma\in G$, any $1\le j\le m$.
\end{theorem}

\begin{lemma} \label{l5.2}
Let $K\subset L$ be $k$-field. If $L$ is retract $K$-rational,
then $K$ is a dense retraction of $L$.
\end{lemma}

\begin{proof}
Let $A$ be an affine $K$-domain whose quotient field is $L$
arising from the definition of retract $K$-rationality. Let
$K[X_1, \cdots,X_n][1/f]$ be the localized polynomial ring and
$\varphi: A\to K[X_1,\ldots,X_n][1/f]$,
$\psi:K[X_1,\ldots,X_n][1/f]$ be the $K$-morphisms satisfying
$\psi \circ \varphi = 1_A$.

Since the singular locus of $A$ defines a non-zero ideal $I$ in
$A$, we may choose any non-zero element $\alpha \in I$; then
replace $A$ by $A[\alpha]$ and replace $K[X_1,\ldots,X_n][1/f]$ by
$K[X_1,\ldots,X_n]$ $[1/(f \phi (r))]$. Thus we may assume that
$A$ is a regular domain from the beginning. For any $r\in
A\backslash\{0\}$, let $g=\varphi(r)$. Find a $K$-morphism
$\Phi:K[X_1,\ldots,X_n][1/f] \rightarrow K$ such that $\Phi(fg)
\neq 0$. Then $\Phi \circ \phi : A \rightarrow K$ is the required
map.
\end{proof}

\bigskip
We consider an application of Theorem \ref{t4.2}.

Recall Theorem \ref{t2.8} provides a criterion of retract
rationality for $K(M)^G$ when $G$ is faithful on $K$ and $M$ is a
$G$-lattice (it is unnecessary to assume that $M$ is a faithful
$G$-lattice). Now we consider the retract rationality for $k(M)^G$
where $G$ acts trivially on the field $k$.

\begin{theorem}\label{t5.3}
Let $G$ be a finite group acting trivially on the field $k$, and
$M$ be a faithful $G$-lattice.

{\rm (i) (Saltman {\cite[Corollary 1.6]{Sa5}})} If $k(M)^G$ is
retract $k$-rational, then $k(G)$ is also retract $k$-rational.

{\rm (ii) (Saltman {\cite[Proposition 1.7]{Sa5}})} If $0
\rightarrow M \rightarrow N \rightarrow E \rightarrow 0$ is an
exact sequence of $G$-lattices where $E$ is invertible, then
$k(N)^G$ is retract $k(M)^G$-rational.
\end{theorem}

We may wonder whether some criterion of retract rationality for
$k(M)^G$ is available. Although we cannot find a complete
solution, we are able to answer this question when $[M]^{fl}$ is
invertible.

\begin{theorem} \label{t5.4}
Let $G$ be a finite group. For any $G$-lattice $M$ in the
following statements, it is assumed that $G$ acts on $k(M)$ by
purely monomial $k$-automorphisms. The following statements are
equivalent, \vspace*{-2mm} \leftmargini=12mm
\begin{enumerate} \itemsep=-1pt
\item[{\rm (i)}] $k(G)$ is retract $k$-rational;
\item[{\rm (ii)}] $k(M)^G$ is retract $k$-rational for some faithful permutation $G$-lattice $M$;
\item[{\rm (iii)}]
$k(M)^G$ is retract $k$-rational for some faithful $G$-lattice $M$ such that $[M]^{fl}$ is invertible;
\item[{\rm (iv)}] $k(M)^G$ is retract $k$-rational for all faithful permutation $G$-lattices $M$;
\item[{\rm (v)}]
$k(M)^G$ is retract $k$-rational for all faithful $G$-lattices $M$ satisfying that $[M]^{fl}$ are invertible;
\item[{\rm (vi)}] $k(M)^G$ is retract $k$-rational for some faithful $G$-lattice $M$.
\end{enumerate}
\end{theorem}

\begin{proof}
${\rm (i)}\Rightarrow{\rm (vi)}$ by taking $M=\bm{Z}[G]$.

${\rm (vi)}\Rightarrow{\rm (i)}$ by Theorem \ref{t5.3}.

The implications ``${\rm (v)}\Rightarrow{\rm (iv)}\Rightarrow{\rm
(ii)}\Rightarrow{\rm (iii)}\Rightarrow{\rm (i)}$" is easy. It
remains to show that ``${\rm (i)}\Rightarrow{\rm (v)}$".

For any faithful $G$-lattice $M$ with $[M]^{fl}$ invertible,
we will show that $k(M)^G$ is retract $k$-rational.

Define $N:=\bm{Z}[G]$ and consider $k(M\oplus N)^G$.

By Theorem \ref{t5.1}, $k(M\oplus N)^G=\{k(M)(z_1,\ldots,z_l)\}^G$
where $\sigma\cdot z_j=z_j$ for any $\sigma\in G$, any $1\le j\le
l=|G|$. Thus $k(M\oplus N)^G=k(M)^G(z_1,\ldots,z_l)$. It follows
that $k(M)^G$ is retract $k$-rational if and only if so is
$k(M\oplus N)^G$ by Lemma \ref{l3.4}.

Now $k(M\oplus N)^G\simeq k(N\oplus M)^G=\{k(N)(M)\}^G$ is retract rational over $k(N)^G$ by Theorem \ref{t2.8}.
Since $k(N)^G=k(G)$ is retract $k$-rational.
Apply Theorem \ref{t4.2}.
We find that $k(M\oplus N)^G$ is retract $k$-rational.

\bigskip
Here is another proof of ``${\rm (i)}\Rightarrow{\rm (v)}$".

Suppose that $k(G)$ is retract $k$-rational and $M$ is a faithful
$G$-lattice with $[M]^{fl}$ invertible.

Let $0\to M\to P\to E\to 0$ be the flabby resolution of $M$ where
$P$ is a permutation lattice and $E$ is an invertible lattice
because $[M]^{fl}$ is invertible. By Theorem \ref{t5.3}, we find
that $k(P)^G$ is retract rational over $k(M)^G$. Thus $k(M)^G$ is
a dense retraction of $k(P)^G$ by Lemma \ref{l5.2}. We may apply
Lemma \ref{l3.4} to show that $k(M)^G$ is retract $k$-rational, if
it is known that $k(P)^G$ is retract $k$-rational. Since
$k(G)=k(\bm{Z}[G]^G)$, we may apply Theorem \ref{t5.1} twice to
$k(P \oplus \bm{Z}[G])^G$ as the preceding proof. Thus we find
that $k(G)$ is retract $k$-rational if and only if so is $k(P)^G$.
Done.
\end{proof}

\begin{idef}{Remarks.}
(i) It is known that $k(M)^G$ is $k$-rational for any $G$-lattice $M$ with $\fn{rank}_{\bm{Z}}(M) \break \le 3$.
See \cite{HK1,HK2,HR}.

(ii) Note that \cite[Theorem 4.3]{Ka2} may be regarded as a hybrid
of Theorem \ref{t2.8} and the above Theorem \ref{t5.4} (with the
help of the following Theorem \ref{t5.5}).
\end{idef}

\begin{theorem} \label{t5.5}
Let $G=\langle \sigma\rangle$ be a cyclic group of order $n=2^rm$
where $2 \nmid m$. The following statements are equivalent,

{\rm (i)} $k(G)$ is retract $k$-rational;

{\rm (ii)} $k(M)^G$ is retract $k$-rational for any $G$-lattice $M$;

{\rm (iii)} Either $\fn{char}k=2$ or $\fn{char}k\ne 2$ such that
$k(\zeta_{2^r})$ is a cyclic extension of $k$.
\end{theorem}

\begin{proof}
The equivalence of (i) and (iii) follows from Theorem \ref{t3.7}.

${\rm (ii)}\Rightarrow{\rm (i)}$ by Theorem \ref{t5.4}.

${\rm (i)}\Rightarrow{\rm (ii)}$ If $M$ is a faithful $G$-lattice,
then $[M]^{fl}$ is invertible by Theorem \ref{t2.5}. Hence
$k(M)^G$ is retract $k$-rational by Theorem \ref{t5.4}.

If $M$ is not faithful, find a normal subgroup $H$ of $G$ so that
$M$ is a faithful lattice over $G/H$. Let $n'=2^sm'$ be the order
of $G/H$ with $2 \nmid m'$. Since $n'\mid n$ and $k(\zeta_{2^r})$
is cyclic over $k$, it follows that $k(\zeta_{2^s})$ is also
cyclic over $k$. Thus $k(G/H)$ is retract $k$-rational by Theorem
\ref{t3.7}. Now we may apply the same arguments in the preceding
paragraph to the group $G/H$.
\end{proof}

We recall a theorem in group theory.

\begin{theorem}[{\cite[Theorem 11, p.175]{Za}}] \label{t5.6}
Let $G$ be a finite group.
Then the following two statements are equivalent,

{\rm (i)} All the Sylow subgroups of $G$ are cyclic;

{\rm (ii)} $G$ is of the form $G=\langle \sigma,\tau\rangle$ with relations $\sigma^m=\tau^n=1$,
$\tau\sigma\tau^{-1}=\sigma^r$ where $m$, $n$, $r$ are positive integers satisfying
\[
\gcd\{(r-1)n,m\}=1 \mbox{ ~ and ~ } r^n\equiv 1 \!\!\pmod{m}.
\]
\end{theorem}

\bigskip
Note that, in the condition (ii) of the above theorem, if $r=1$,
it is understood as ``$\gcd\{n,m\}=1$".

The following result is an extension of Theorem \ref{t5.5}. We
choose to formulate only one direction among the various
directions of implication.

\begin{theorem} \label{t5.7}
Let $G$ be a finite group satisfying the property in Theorem
\ref{t5.6}. If $k(G)$ is retract $k$-rational, then $k(M)^G$ is
retract $k$-rational for any $G$-lattice $M$.
\end{theorem}

\begin{proof}
By the same method as in the proof of Theorem \ref{t5.5}, we may
assume that $M$ is faithful. Then apply Theorem \ref{t2.5} for
such a group $G$, and use Theorem \ref{t5.4}.
\end{proof}

Now we consider an application of Theorem \ref{t5.5}.

Recall two previous results about the rationality problem and unramified Brauer groups.

\begin{theorem}[Kang {\cite[Theorem 1.4]{Ka1}}] \label{t5.8}
Let $k$ be a field and $G$ be a finite group.
Assume that {\rm (i)} $G$ contains an abelian normal subgroup $H$ so that $G/H$ is cyclic of order $n$,
{\rm (ii)} $\bm{Z}[\zeta_n]$ is a unique factorization domain, and
{\rm (iii)} $\zeta_e\in k$ where $e$ is the exponent of $G$.
If $G\to GL(V)$ is any finite-dimensional linear representation of $G$ over $k$,
then $k(V)^G$ is rational over $k$.
In particular, $k(G)$ is $k$-rational.
\end{theorem}

\begin{theorem}[Bogomolov {\cite[Lemma 4.9]{Bo}}] \label{t5.9}
Let $G$ be a finite group containing an abelian normal subgroup
$H$ such that $G/H$ is cyclic. Then
$\fn{Br}_{v,\bm{C}}(\bm{C}(G))=0$.
\end{theorem}

What we will prove next is that, with the same assumptions as in
Theorem \ref{t5.9}, $\bm{C}(G)$ is retract $\bm{C}$-rational.
Hence it is not surprising that $\fn{Br}_{V,\bm{C}}(\bm{C}(G))=0$
in this situation.

\begin{theorem} \label{t5.10}
Let $k$ be an infinite field and $G$ be a finite group. Assume
that {\rm (i)} $G$ contains an abelian normal subgroup $H$ so that
$G/H$ is cyclic, and {\rm (ii)} $\zeta_{e'} \in k$ with $e'=lcm \{
exp \, (H), ord (\tau) \}$ where $\tau$ is some element in $G$ and
the image of $\tau$ in $G/H$ generates the cyclic group $G/H$. If
$G \rightarrow GL(V)$ is any linear representation of $G$ on the
$k$-vector space $V$, then $k(V)^G$ is retract $k$-rational. In
particular, $k(G)$ is retract $k$-rational.
\end{theorem}

\begin{proof}
Step 1. We will go over the proof of Theorem \ref{t5.8} in the
paper \cite{Ka1}. By \cite[Corollary 3.2]{Ka1}, the proof of
Theorem \ref{t5.8} is valid under the weaker assumption on
$\zeta_{e'}$. We will show that
$k(V)^G=k(M)^{\pi}(Y_1,\ldots,Y_r)$ where
$\pi=G/H=\langle\bar{\tau}\rangle$ and $M$ is a $\pi$-lattice.

Note that the assumption that $\bm{Z}[\zeta_n]$ is a unique
factorization domain is used in the proof of \cite[Theorem
2.2]{Ka1}. This theorem asserts that $k(M)$ is $\pi$-isomorphic to
$k(L)$, a fact which appears only in Step 5 of the proof of
\cite[Theorem 1.4, line 7 from the bottom on page 1218]{Ka1}.

On the other hand, in Step 4 of the proof of \cite[Theorem
1.4]{Ka1}, it is known that $k(V)^G=k(y(i,j):1\le i\le r,~ 1\le
j\le d_i-1)^G (Y_1,\ldots,Y_r)=k(z(i,j):1\le i\le r$, $1\le j\le
d_i-1)^G(Y_1,\ldots,Y_r)$ where $G=\langle H,\tau \rangle$ acts on
these $z(i,j)$ by
\begin{equation}
\begin{aligned}
\tau:~& z(i,1)\mapsto z(i,2)\mapsto \cdots\mapsto z(i,d_i-1)\mapsto
\left(\prod_{1\le j\le d_i-1}z(i,j)\right)^{-1}, \\
\sigma:~& z(i,j)\mapsto \Psi_i(\tau^{-(j-1)}\sigma\tau^{j-1}) z(i,j)
\end{aligned} \label{eq1}
\end{equation}
where $\sigma \in H$ and $1\le j\le d_i-1$.

The first two paragraphs of Step 5 of the proof of \cite[Theorem 1.4]{Ka1} shows that
$k(z(i,j):1\le i\le r,~1\le j\le d_i-1)^H=k(M)$.
Hence $k(V)^G=k(M)^{\langle \tau\rangle}(Y_1,\ldots,Y_r)$.
From Formula \eqref{eq1},
it is clear that $\tau$ acts on $k(M)$ by purely monomial $k$-automorphisms.

\bigskip
Step 2. By Fischer's Theorem \cite[Theorem 6.1; KP, Corollary
1.5]{Sw1}, $k(G/H)$ is $k$-rational; thus it is retract
$k$-rational. Applying Theorem \ref{t5.5}, we find that
$k(M)^{\langle \tau\rangle}$ is retract $k$-rational. By Lemma
\ref{l3.4}, $k(V)^G$ is retract $k$-rational.

In particular, take a $k$-vector space $V$ whose dual space is
equal to $\bigoplus_{g\in G} k\cdot x(g)$, the regular
representation of $G$. We find that $k(G)=k(V)^G$ is retract
$k$-rational.
\end{proof}

\begin{remark}
Compare Theorem \ref{t5.10} with Proposition 5.2 in \cite{Ka2}
(and also Theorem 1.11, Theorem 1.12 and Corollary 5.1 there).
There the assumption $\zeta_{e'}\in k$ is waived, while other
assumptions, e.g.\ the group extension $1\to H\to G\to C_n\to 1$
splits and the structures of some Galois extensions over $k$, are
required.
\end{remark}

\section{Monomial actions}

Recall the definition of the fixed field $k_{\alpha}(M)^G$ of a
monomial action of $G$ (see Definition \ref{d2.2}). Throughout
this section, $G$ acts trivially on $k$. We will generalize the
following theorem of Barge.

\begin{theorem}[Barge {\cite[Theorem IV-1]{Ba}}] \label{t6.1}
Let $G$ be a finite group. The following two statements are
equivalent,

{\rm (i)} All the Sylow subgroups of $G$ are cyclic;

{\rm (ii)} $Br_{v,\bm{C}}(\bm{C}_{\alpha}(M)^G)=0$ for all
$G$-lattices $M$, for all short exact sequences of
$\bm{Z}[G]$-modules $\alpha : 0 \rightarrow \bm{C}^{\times}
\rightarrow M_{\alpha} \rightarrow M \rightarrow 0$.
\end{theorem}

First we recall an $H^1$ trivial embedding theorem due to Saltman.

\begin{theorem}[Saltman {\cite[Proposition 2]{Sa7}}] \label{t6.2}
Let $G$ be a finite group, $M$ be a $G$-lattice. If $\alpha:0\to
k^{\times}\to M_\alpha\to M\to 0$ is an exact sequence of
$\bm{Z}[G]$-modules, then there is an exact sequence $\beta:0\to
k^{\times}\to N_\beta\to N\to 0$ satisfying that {\rm (i)} $N$ is
a $G$-lattice, {\rm (ii)} $M_\alpha \subset N_\beta$, {\rm (iii)}
$N_\beta$ is $H^1$ trivial, i.e.\ $H^1(G',N_\beta)=0$ for any
subgroup $G'\subset G$, and {\rm (iv)} $N_\beta/M_\alpha$ is a
permutation $G$-lattice.
\end{theorem}

\begin{theorem} \label{t6.3}
Let $G$ be a finite group. Then $k(G)$ is retract $k$-rational if
and only if $k_\alpha(M)^G$ is retract $k$-rational for any
invertible $G$-lattice $M$, for any short exact sequence of
$\bm{Z}[G]$-modules $\alpha:0\to k^{\times}\to M_\alpha\to M\to 0$
with $G$ acting faithfully on $M_\alpha$.
\end{theorem}

\begin{proof}
It suffices to show that ``only if" part.

Suppose that $k(G)$ is retract $k$-rational and $\alpha:0\to
k^{\times}\to M_\alpha\to M\to 0$ is the given extension.

Choose a $G$-lattice $N$ such that $M\oplus N$ is a permutation
$G$-lattice. Denote $P=M\oplus N$.

We extend the action of $G$ from $k_\alpha(M)$ to
$k_\alpha(M\oplus N)$ by requiring $G$ acts on $k(N)$ by purely
monomial $k$-automorphisms. Then $G$ acts faithfully on
$k_\alpha(M\oplus N)$.

Since $M\oplus N=P$, it follows that $G$ acts on
$k_\alpha(P)=k_\alpha(M\oplus N)$ by monomial $k$-automorphisms.
Moreover, if $P=\bigoplus_{1\le i\le n}\bm{Z}\cdot
x_i$\vspace*{2pt} and $G$ permutes $\{x_i:1\le i\le n\}$, then $G$
acts on $k_\alpha(P)=k(x_1,\ldots,x_n)$ by linear
$k$-automorphisms, i.e. for any $\sigma \in G$, any $1 \leq i \leq
n$, $\sigma \cdot x_i = a_i(\sigma) x_j$ where $j$ depends on $i$
and $a_i(\sigma)$ is some non-zero element in $k$ depending on
$\sigma$ and $i$.

Consider $k_\alpha(P\oplus Q)^G$ where $Q=\bm{Z}[G]$. By Theorem
\ref{t5.1}, $k_\alpha(P\oplus Q)^G$ is rational over
$k_\alpha(P)^G$; apply the same theorem again, $k_\alpha(P\oplus
Q)^G \simeq k_\alpha(Q\oplus P)^G$ is rational over
$k_\alpha(Q)^G=k(G)$. Since $k(G)$ is retract $k$-rational, so is
$k_\alpha(P)^G$ by Lemma \ref{l3.4}.

On the other hand, consider $k_\alpha(M\oplus N)^G$ ($\simeq
k_\alpha(P)^G$). By Lemma \ref{l3.4}, $k_\alpha(M)^G$ is a dense
retraction of $k_\alpha(M\oplus N)^G$ (note that $G$ acts
faithfully on $k_\alpha(M)$). Since $k_\alpha(P)^G$ is retract
$k$-rational, so is $k_\alpha(M)^G$ again by Lemma \ref{l3.4}.
\end{proof}

\begin{lemma} \label{l6.4}
Let $G$ be a finite group. Assume that $k(\bar{G})$ is retract
$k$-rational for all quotient groups $\bar{G}$ of the group $G$.
Let $M$ be a $G$-lattice and $\alpha:0 \rightarrow k^{\times}\to
M_\alpha\to M\to 0$ be a short exact sequence of
$\bm{Z}[G]$-modules satisfying that
\begin{enumerate}
\item[{\rm (i)}] denoting $H= \{ \sigma \in G : \sigma$ acts
trivially on $M_{\alpha} \}$, then there is a short exact sequence
of $\bm{Z}[G/H]$-modules $\beta:0\to k^{\times}\to N_\beta\to N\to
0$ where, regarding $N_\beta$ and $N$ as $\bm{Z}[G]$-modules, $N$
is a $G$-lattice, $N_\beta$ is $H^1$ trivial, $M_{\alpha}\subset
N_\beta$, and $N_\beta/M_{\alpha}$ is a permutation $G$-lattice;
\item[{\rm (ii)}] $N$ is $H^1$ trivial; and \item[{\rm (iii)}]
$[M]^{fl}$ is an invertible $G$-lattice.
\end{enumerate}
Then $k_\alpha(M)^G$ is retract $k$-rational.
\end{lemma}

\begin{remark}The assumption (i) can be achieved
by Theorem \ref{t6.2}. On the other hand, the assumption (ii) is
essential. In fact, Saltman proves that, if $k$ is an infinite
filed with $\fn{char}k\ne 2$ and $\sigma: k(x,y,z)\to k(x,y,z)$ is
a $k$-automorphism defined by $\sigma(x)=a/x$, $\sigma(y)=b/y$,
$\sigma(z)=c/z$ where $a,b,c\in k\backslash\{0\}$ satisfying
$[k(\sqrt{a},\sqrt{b},\sqrt{c}):k]=8$, then $k(x,y,z)^{\langle
\sigma\rangle}$ is not retract $k$-rational \cite{Sa8}. The above
theorem is not applicable to Saltman's example, because $N$ is not
$H^1$ trivial for any embedding of $M_\alpha$ into an $H^1$
trivial module $N_\beta$.
\end{remark}

\begin{proof}
Replace the group $G$ by $G/H$ where $H$ is the subgroup in the
assumption (i). We may assume that the $G$-module $M_{\alpha}$ is
faithful.

Let $N_\beta$ be any $H^1$ trivial embedding of $M_{\alpha}$
satisfying the assumptions (i), (ii) and (iii).

Since $N_\beta/M_\alpha$ is a permutation $G$-lattice,
$k_\beta(N_\beta)=k_\alpha(M_\alpha)(x_1,\ldots,x_n)$ for some
$x_1,\ldots,x_n$ satisfying that, for any $\sigma\in G$,
$\sigma(x_i)=a_i(\sigma)\cdot x_j$ for some $x_j$ and some
$a_i(\sigma)\in k_\alpha(M_\alpha)\backslash\{0\}$. By Theorem
\ref{t5.1}, $k_\beta(N_\beta)^G$ is rational over
$k_\alpha(M_\alpha)^G$. By Lemma \ref{l3.4}, $k_\beta(N_\beta)^G$
is retract $k$-rational if and only if so is
$k_\alpha(M_\alpha)^G$.

From the snake lemma of the following diagram
\[ \arraycolsep=2pt
\begin{array}{ccccccc}
0 \to & k^{\times} &\to & M_\alpha & \to & M & \to 0 \\[2pt]
& \| & & \downarrow & & \downarrow & \\[2pt]
0 \to & k^{\times} & \to & N_\beta & \to & N & \to 0
\end{array}
\]
we find that $N_\beta/M_\alpha \simeq N/M$ is a permutation
$G$-lattice. By \cite[Lemma 3.1]{Sw2}, $[N]^{fl}=[M]^{fl}$ is an
invertible lattice.

Let $0\to N\to P\to N'\to 0$ be a flabby resolution of $N$, i.e.\
$P$ is a permutation lattice and $N'$ is a flabby lattice. By
Lemma \ref{l2.4}, this short exact sequence splits, i.e.\ $P\simeq
N\oplus N'$. Hence $N$ is an invertible $G$-lattice. Thus
$k_\beta(N_\beta)^G$ is retract $k$-rational by Theorem
\ref{t6.3}.
\end{proof}

\begin{lemma} \label{l6.5}
Let $G$ be a finite group satisfying the property in Theorem
\ref{t5.6} and $k(\bar{G})$ is retract $k$-rational for all
quotient groups $\bar{G}$ of the group $G$. Let $\alpha:0\to
k^{\times}\to M_\alpha\to M\to 0$ be a short exact sequence of
$\bm{Z}[G]$-modules where $M$ is a $G$-lattice. Assume the exact
sequence $\alpha$ satisfies assumptions (i) and (ii) in Lemma
\ref{l6.4}. Then $k_\alpha(M)^G$ is retract $k$-rational.
\end{lemma}

\begin{proof}
By Lemma \ref{l6.4}, it remains to show that the assumption (iii)
is valid for $\alpha$, i.e.\ $[M]^{fl}$ is invertible. But this
follows from Theorem \ref{t2.5} and Theorem \ref{t5.6}.
\end{proof}

\begin{remark}
Lemma \ref{l6.5} was proved by Saltman when $G\simeq C_p$ where
$p$ is a prime number and $\zeta_p\in k$ \cite[Lemma 11]{Sa7}.
\end{remark}

\medskip
The next result is a generalization of Theorem \ref{t6.1} and is
valid for any field  $k$ which is algebraically closed and
$\fn{char}k\nmid \, \mid G \mid$. But we choose to present our
result when $k$ is the field of complex numbers.

\begin{theorem} \label{t6.6}
Let $G$ be a finite group. The following three statements are
equivalent,

{\rm (i)} All the Sylow subgroups of $G$ are cyclic;

{\rm (ii)} $\bm{C}_{\alpha}(M)^G$ is retract $\bm{C}$-rational for
all $G$-lattices $M$, for all short exact sequences of
$\bm{Z}[G]$-modules $\alpha : 0 \rightarrow \bm{C}^{\times}
\rightarrow M_{\alpha} \rightarrow M \rightarrow 0$;

{\rm (iii)} $Br_{v,\bm{C}}(\bm{C}_{\alpha}(M)^G)=0$ for all
$G$-lattices $M$, for all short exact sequences of
$\bm{Z}[G]$-modules $\alpha : 0 \rightarrow \bm{C}^{\times}
\rightarrow M_{\alpha} \rightarrow M \rightarrow 0$.
\end{theorem}

\begin{proof}
${\rm (ii)}\Rightarrow{\rm (iii)}$ by Theorem \ref{t3.2}.

${\rm (iii)}\Rightarrow{\rm (i)}$ by Theorem \ref{t6.1}.

It remains to show that ${\rm (i)}\Rightarrow{\rm (ii)}$. We will
apply Lemma \ref{l6.5}.

Let $G$ be a finite group satisfying the assumption (i) of this
theorem. By Theorem \ref{t5.6}, the group $G$ and all of its
quotient groups are metacyclic; thus Theorem \ref{t5.10} is
applicable to these groups. It follows that $\bm{C}(\bar{G})$ is
retract $\bm{C}$-rational for all quotient groups $\bar{G}$ of the
group $G$.

For a short exact sequence of $\bm{Z}[G]$-modules $\alpha : 0
\rightarrow \bm{C}^{\times} \rightarrow M_{\alpha} \rightarrow M
\rightarrow 0$, we will show that $\bm{C}_{\alpha}(M)^G$ is
retract $\bm{C}$-rational. Replacing $G$ by some quotient group
$G/H$ if necessary, we may assume that $G$ acts faithfully on
$M_{\alpha}$.

In order to apply Lemma \ref{l6.5}, we should check the validity
of the assumptions (i) and (ii) of Lemma \ref{l6.5}. The
assumption (i) is valid by Theorem \ref{t6.2}. As to the
assumption (ii), we will show that $H^2(G', \bm{C}^{\times})
\rightarrow H^2(G', N_{\beta})$ is injective for any subgroup $G'
\subset G$, which is equivalent to the assumption (ii) of Lemma
\ref{l6.5}, because $N_{\beta}$ is $H^1$ trivial.

Note that $H^2(G', \bm{C}^{\times})$ is the trivial group, because
we may consider $H^2(G'_p, \bm{C}^{\times})$ where $G'_p$ is a
$p$-Sylow subgroup of $G'$ and we find that $H^2(G'_p,
\bm{C}^{\times}) \simeq H^0(G'_p, \bm{C}^{\times}) \simeq
\bm{C}^{\times}/ (\bm{C}^{\times})^q = 0$ where $q$ is the order
of the cyclic group $G'_p$. Hence the result.
\end{proof}

\newpage
\renewcommand{\refname}{\centering{References}}

\end{document}